\documentclass[11pt,leqno]{amsart}

\usepackage{latexsym,amssymb}

\theoremstyle{plain} 

\newtheorem{theorem}{Theorem}[section]
\newtheorem*{Theorem B}{Theorem B}
\newtheorem*{Theorem A}{Theorem A}

\newtheorem{lemma}{Lemma}[section]

\theoremstyle{remark} 
\newtheorem{remark}{Remark}[section]
\theoremstyle{remark}

\theoremstyle{definition}

\numberwithin{equation}{section}
\def\<{\left < }
\def\>{\right >}
\def\({\left ( }
\def\){\right )}

\def\R{\mathbb{R}}
\def\C{\mathbb{C}}

\begin{document}

\title[Curvature inequalities]{Curvature inequalities for Lagrangian submanifolds:\\ the final solution}

\author[B.-Y. Chen]{Bang-Yen Chen}
\address{Department of Mathematics, Michigan State University, East Lansing, Michigan 48824-1027,
USA} \email{bychen@math.msu.edu}

\author[F. Dillen]{Franki Dillen}
\address{KU Leuven\\ Departement
Wiskunde\\ Celestijnenlaan 200B -- Box 2400\\ BE-3001 Leuven\\
Belgium} \email{franki.dillen@wis.kuleuven.be}

\author[J. Van der Veken]{Joeri Van der Veken}
\address{KU Leuven\\ Departement
Wiskunde\\ Celestijnenlaan 200B -- Box 2400\\ BE-3001 Leuven\\
Belgium} \email{joeri.vanderveken@wis.kuleuven.be}
\thanks{The third author is a post-doctoral researcher supported by the Research Foundation -- Flanders (F.W.O.).}

\author[L. Vrancken]{Luc Vrancken}
\address{Universit\'e de Valenciennes, Lamath, ISTV2, Campus du Mont Houy, 59313 Valenciennes,
Cedex 9, France; KU Leuven\\ Departement
Wiskunde\\ Celestijnenlaan 200B -- Box 2400\\ BE-3001 Leuven\\
Belgium} \email{luc.vrancken@univ-valenciennes.fr}

\begin{abstract}
Let $M$ be an $n$-dimensional Lagrangian submanifold of a complex space form. We prove a pointwise inequality $$\delta(n_1,\ldots,n_k) \leq a(n,k,n_1,\ldots,n_k) \|H\|^2 + b(n,k,n_1,\ldots,n_k)c,$$ with on the left hand side any delta-invariant of the Riemannian manifold $M$ and on the right hand side a linear combination of the squared mean curvature of the immersion and the constant holomorphic sectional curvature of the ambient space. The coefficients on the right hand side are optimal in the sence that there exist non-minimal examples satisfying equality at at least one point. We also characterize those Lagrangian submanifolds satisfying equality at any of their points. Our results correct and extend those given in \cite{CD}.
\end{abstract}



\maketitle

\section{Introduction}

Let $M^n$ be an $n$-dimensional Riemannian manifold. In the 1990ties, the first author introduced a new family of curvature functions on $M^n$, the so-called \emph{delta-invariants}. In particular, for any $k$ integers $n_1,\ldots,n_k$ satisfying $$2 \leq n_1 \leq \ldots \leq n_k \leq n-1\;\; and \;\;n_1 + \ldots + n_k \leq n,$$ a delta-invariant $\delta(n_1,\ldots,n_k)$ was defined at any point of $M^n$.

In a K\"ahler manifold with complex structure $J$, a special role is played by Lagrangian submanifolds. These are submanifolds for which $J$ maps the tangent space into the normal space and vice versa at any point. In \cite{CDVV1} and \cite{CDVV2} the following pointwise inequality for a Lagrangian submanifold $M^n \hookrightarrow \tilde M^n(4c)$ of a complex space form of constant holomorphic sectional curvature $4c$ was obtained:
\begin{equation}\begin{aligned}& \delta(n_1,\ldots,n_k) \leq  \dfrac{n^2 \left( n+k+1-\sum_{i=1}^kn_i \right)}{2 \left( n+k-\sum_{i=1}^kn_i \right)} \|H\|^2\\&\hskip.6in + \dfrac 12 \left( n(n-1) - \sum_{i=1}^k n_i(n_i-1) \right) c.\end{aligned} \end{equation}
Here, $H$ is the mean curvature vector of the immersion at the point under consideration. The importance of this type of inequalities is that the left hand side is intrinsic, i.e., it only depends on $M^n$ as a Riemannian manifold itself, whereas the right hand side contains extrinsic information, i.e., depending of the immersion under consideration. For example, the inequality shows that a necessary condition for a Riemannian manifold $M^n$ to allow a minimal Lagrangian immersion into $\C^n$ is that all the delta-invariants at all points are non-positive. However, it was proven in \cite{C1} that if equality holds in the above inequality at some point, the mean curvature of the immersion has to vanish at this point. This suggests that the inequality is not optimal, i.e., that the coefficient of $\|H\|^2$ can be replaced by a smaller value. The following improvement was given in \cite{CD}:
\begin{equation}\begin{aligned} \label{0.2}\delta(n_1,\ldots,n_k) \leq &\frac{n^2 \left(n -\sum_{i=1}^k n_i + 3k - 1 - 6\sum_{i=1}^k \frac{1}{2+n_i} \right)}
{2 \left(n -\sum_{i=1}^k n_i + 3k + 2 - 6\sum_{i=1}^k \frac{1}{2+n_i} \right)} \|H\|^2
\\&\hskip.0in +\frac{1}{2}\left(n(n-1)-\sum_{i=1}^k n_i(n_i-1)\right)c.\end{aligned} \end{equation}
It was pointed out in \cite{CD2} that the proof of inequality \eqref{0.2} given \cite{CD} is incorrect when $\sum_{i=1}^k \frac{1}{2+n_i}>\frac{1}{3}$. 

The purpose of this paper is two-fold. First, we correct the proof of the above inequality in the case $n_1+\ldots+n_k<n$ (Theorem \ref{theo1}) and then we show that the inequality can be improved in the case $n_1+\ldots+n_k=n$ (Theorem \ref{theo2}). In both cases, we also characterize those Lagrangian submanifolds attaining equality at any of their points, thereby showing that the inequalities are optimal, in the sense that non-minimal examples occur.


\section{preliminaries}

Let us first recall the definition of the delta-invariants. Let $p$ be a point of an $n$-dimensional Riemannian manifold $M^n$ and let $L$ be a linear subspace of the tangent space $T_pM^n$. If $\{e_1,\ldots,e_{\ell}\}$ is an orthonormal basis of $L$, we define
$$ \tau(L) := \sum_{i,j=1 \atop i<j}^{\ell} K(e_i \wedge e_j), $$
where $K(e_i \wedge e_j)$ denotes the sectional curvature of the plane spanned by $e_i$ and $e_j$. Remark that the right hand side is indeed independent of the chosen orthonormal basis and that $\tau := \tau(T_pM^n)$ is nothing but the scalar curvature of $M^n$ at $p$. Now let $n_1,\ldots,n_k$ be integers such that $$2 \leq n_1 \leq \ldots \leq n_k \leq n-1 \;\; and \;\; n_1 + \ldots + n_k \leq n,$$ then the delta-invariant $\delta(n_1,\ldots,n_k)$ at the point $p$ is defined as follows:
$$ \delta(n_1,\ldots,n_k)(p) := \tau - \inf \left\{ \sum_{i=1}^k \tau(L_i) \right\}, $$
where the infimum is taken over all $k$-tuples $(L_1,\ldots,L_k)$ of mutually orthogonal subspaces of $T_pM^n$ with $\dim(L_i)=n_i$ for $i=1,\ldots,k$. Due to a compactness argument, the infimum is actually a minimum. Remark that the simplest delta-invariant is $\delta(2) = \tau - \inf\{K(\pi)\ |\ \pi \mbox{ is a plane in } T_pM^n\}$.

In the following we will denote by $\tilde M^n(4c)$ a complex space form of complex dimension $n$ and constant holomorphic sectional curvature $4c$. Let $M^n \hookrightarrow \tilde M^n(4c)$ be a Lagrangian immersion and denote the Levi-Civita connections of $M^n$ and $\tilde M^n(4c)$ by $\nabla$ and $\tilde \nabla$ respectively. If $X$ and $Y$ are vector fields on $M^n$, then the formula of Gauss gives a decomposition of $\tilde\nabla_XY$ into its components tangent and normal to $M^n$:
$$ \tilde\nabla_X Y = \nabla_X Y + h(X,Y), $$ 
defining in this way the second fundamental form $h$, a symmetric $(1,2)$-tensor field taking values in the normal bundle. The mean curvature vector field is defined as 
$$ H := \frac 1n \mbox{trace}\hskip.01in(h). $$ 
An important property of Lagrangian submanifolds is that the cubic form, i.e., the $(0,3)$-tensor field on $M^n$ defined by $\langle h(\cdot,\cdot),J\cdot \rangle$, where $J$ is the almost complex structure of $\tilde M^n(4c)$, is totally symmetric. Finally, we recall the equation of Gauss: if $R$ is the Riemann-Christoffel curvature tensor of $M^n$ and $X$, $Y$, $Z$ and $W$ are tangent to $M^n$, then
\begin{equation}\begin{aligned} & \langle R(X,Y)Z,W \rangle = \langle h(X,W),h(Y,Z) \rangle - \langle h(X,Z),h(Y,W) \rangle \\&\hskip.7in + c \left( \langle X,W \rangle \langle Y,Z \rangle - \langle X,Z \rangle \langle Y,W \rangle \right). \end{aligned}\end{equation}

The following result on the existence of Lagrangian submanifolds can be found for example in \cite{C2}.

\begin{lemma} \label{lem1}
For any set of real numbers $\{a_{ABC}\ |\ A,B,C = 1,\ldots,n\}$, which is symmetric in the three indices $A$, $B$ and $C$, there exists a Lagrangian immersion $F:U\subseteq\R^n \to \mathbb C^n$ and a point $p \in U$ such that the second fundamental form $h$ of $F$ at $p$ is given by $\langle h(e_A,e_B),J F_{\ast} e_C \rangle = a_{ABC}$, where $\{e_1,\ldots,e_n\}$ is the standard basis of $\R^n$ and $J$ is the standard complex structure of $\C^n$.
\end{lemma}

\begin{proof}
Let $f:U \subseteq \R^n \to \R : (x_1,\ldots,x_n) \mapsto f(x_1,\ldots,x_n)$ be a smooth function on an open subset $U$ of $\R^n$. Then one can verify that $F : U \subseteq \R^n \to \C^n : (x_1,\ldots,x_n) \mapsto (x_1+if_{x_1},\ldots,x_n+if_{x_n})$ is a Lagrangian immersion satisfying $\langle h(e_A,e_B),JF_{\ast}e_C \rangle = f_{x_A x_B x_C}$ at every point of $U$. Here, an index $x_j$ means partial differentiation with respect to $x_j$.

For a given set of real numbers $\{a_{ABC}\ |\ A,B,C = 1,\ldots,n\}$, which is symmetric in the three indices, one can easily construct a smooth function $f$, a degree $3$ polynomial for example, which satisfies $f_{x_A x_B x_C} = a_{ABC}$ and one can define $F$ as above.
\end{proof}

We end this section by stating a fact of elementary linear algebra, which will be useful in the proof or our main results.

\begin{lemma} \label{lem2}
For real numbers $A_1,\ldots,A_k$, denote by
$\Delta(A_1,\ldots,A_k)$ the determinant of the matrix with $A_1,
\ldots, A_k$ on the diagonal and all other entries equal to $1$:
$$ \Delta(A_1,\ldots,A_k) = \left| \begin{array}{ccccc} A_1&1&\cdots&1&1 \\ 1&A_2&\cdots&1&1 \\ \vdots&\vdots&\ddots&\vdots&\vdots \\ 1&1&\cdots&A_{k-1}&1 \\ 1&1&\cdots&1&A_k\end{array}\right|.$$
Then
$$ \Delta(A_1,\ldots,A_k) = \prod_{i=1}^k (A_i-1) + \sum_{i=1}^k \prod_{j \neq i} (A_j-1). $$
In particular, if none of the numbers $A_1,\ldots,A_k$ equals $1$,
then
$$ \Delta(A_1,\ldots,A_k) = \left(1+\frac{1}{A_1-1}+ \ldots +\frac{1}{A_k-1}\right)(A_1-1)\ldots(A_k-1).$$
\end{lemma}

\begin{proof}
The result is true for $k=1$ and $k=2$. Now assume that $k\geq3$ and let $A_1,\ldots,A_k$ be arbitrary real
numbers. We will compute the determinant $\Delta(A_1,\ldots,A_k)$ by
first replacing the $k$th column by the $k$th column
minus the $(k-1)$th column, then replacing the $k$th row by the
$k$th row minus the $(k-1)$th row and finally developing the
determinant with respect to the last column:
\begin{align*}
\Delta(A_1,\ldots,A_k) &= \left| \begin{array}{ccccc} A_1&1&\cdots&1&0 \\ 1&A_2&\cdots&1&0 \\ \vdots&\vdots&\ddots&\vdots&\vdots \\ 1&1&\cdots&A_{k-1}&1-A_{k-1} \\ 1&1&\cdots&1&A_k-1\end{array} \right|
\\
\\
&= \left| \begin{array}{ccccc} A_1&1&\cdots&1&0 \\ 1&A_2&\cdots&1&0 \\ \vdots&\vdots&\ddots&\vdots&\vdots \\ 1&1&\cdots&A_{k-1}&1-A_{k-1} \\ 0&0&\cdots&1-A_{k-1}&A_k+A_{k-1}-2\end{array} \right| \\
\\&\hskip-.4in=(A_k+A_{k-1}-2)\Delta(A_1,\ldots,A_{k-1})-(A_{k-1}-1)^2\Delta(A_1,\ldots,A_{k-2}).
\end{align*}
It is now sufficient to verify that the expression for
$\Delta(A_1,\ldots,A_k)$ given in the statement of the lemma indeed
satisfies the recursion relation
\begin{equation}\begin{aligned}\notag & \Delta(A_1,\ldots,A_k)=(A_k+A_{k-1}-2)\Delta(A_1,\ldots,A_{k-1})\\& \hskip.5in-(A_{k-1}-1)^2\Delta(A_1,\ldots,A_{k-2}),\end{aligned} \end{equation}
with the initial conditions $\Delta(A_1)=A_1$ and
$\Delta(A_1,A_2)=A_1A_2-1$. This can be done by a straightforward
computation.
\end{proof}

\section{The main results}

Before stating our main theorems, we will introduce some notations. For a given delta-invariant $\delta(n_1,\ldots,n_k)$ on a Riemannian manifold $M^n$ (with $2 \leq n_1 \leq \ldots \leq n_k \leq n-1$ and $n_1 + \ldots + n_k \leq n$) and a point $p \in M^n$, we consider mutually orthogonal subspaces $L_1, \ldots, L_k$ with $\dim(L_i) = n_i$ of $T_pM^n$, minimizing the quantity $\tau(L_1)+ \ldots +\tau(L_k)$. We then choose an orthonormal basis $\{e_1,\ldots,e_{n}\}$ for $T_pM^n$ such that
\begin{align*}
& e_1,\ldots,e_{n_1} \in L_1, \\
& e_{n_1+1},\ldots,e_{n_1+n_2} \in L_2, \\
& \ \vdots \\
& e_{n_1+\ldots+n_{k-1}+1},\ldots,e_{n_1+\ldots+n_k} \in L_k,
\end{align*}
and we define
\begin{align*}
& \Delta_1 := \{1,\ldots, n_1\}, \\
& \Delta_2 := \{n_1+1,\ldots, n_1+n_2\}, \\
& \ \vdots \\
& \Delta_k := \{n_1+\ldots+n_{k-1}+1,\ldots,n_1+\ldots+n_k\}, \\
& \Delta_{k+1} := \{n_1+\ldots+n_k+1,\ldots,n\}.
\end{align*}
From now on, we will use the following conventions for the ranges of
summation indices:
$$ A,B,C\in\{1,\ldots,n\}, \quad i,j\in\{1,\ldots,k\}, \quad
\alpha_i,\beta_i\in\Delta_i, \quad r,s\in\Delta_{k+1}.$$
Finally, we define $n_{k+1} := n-n_1-\ldots-n_k$. Remark that this may be zero, in which case $\Delta_{k+1}$ is empty.
We shall denote the components of the
second fundamental form by $h^C_{AB}=\langle
h(e_A,e_B),Je_C\rangle$. Due to the symmetry of the
cubic form, these are symmetric with respect to the three indices
$A$, $B$ and $C$.

\begin{theorem} \label{theo1}
Let $M^n$ be a Lagrangian submanifold of a complex space form
$\tilde M^n(4c)$. Let $n_1,\ldots,n_k$ be integers satisfying $2\leq
n_1 \leq \ldots \leq n_k \leq n-1$ and $n_1+\ldots+n_k < n$. Then, at any point of $M^n$, we have
\begin{equation}\begin{aligned}\notag &\delta(n_1,\ldots,n_k)
\leq \frac{n^2 \left(n -\sum_{i=1}^k n_i + 3k - 1 - 6\sum_{i=1}^k \frac{1}{2+n_i} \right)}
{2 \left(n -\sum_{i=1}^k n_i + 3k + 2 - 6\sum_{i=1}^k \frac{1}{2+n_i} \right)} \|H\|^2
\\& \hskip1.0in +\frac{1}{2}\left(n(n-1)-\sum_{i=1}^k n_i(n_i-1)\right)c.
\end{aligned} \end{equation}

Assume that equality holds at a point $p \in M^n$. Then with the choice of basis and the notations introduced at the beginning of this section, one has
\begin{itemize}
\item $h^A_{BC}=0$ if $A,B,C$ are mutually different and not all in the same $\Delta_i$ ($i=1,\ldots,k$),
\item $\displaystyle{ h^{\alpha_i}_{\alpha_j \alpha_j} = h^{\alpha_i}_{rr} = \sum_{\beta_i \in \Delta_i} h^{\alpha_i}_{\beta_i \beta_i}=0 }$ for $i \neq j$,
\item $h^r_{rr} = 3 h^r_{ss} = (n_i+2) h^r_{\alpha_i \alpha_i}$ for $r \neq s$.
\end{itemize}
\end{theorem}

\begin{proof} The proof consists of four steps.
\smallskip

\noindent\emph{Step 1: Set-up.} Fix a delta-invariant $\delta(n_1,\ldots,n_k)$ and a point $p\in M^n$. Take linear subspaces $L_1,\ldots,L_k$ of $T_pM^n$ and and orthonormal basis $\{e_1,\ldots,e_n\}$ of $T_pM^n$ as described above. From the equation of Gauss we obtain that
\begin{align*}
& \tau= \frac{n(n-1)}{2}c + \sum_{A} \sum_{B<C} (h^A_{BB}h^A_{CC}-(h^A_{BC})^2),\\
& \tau(L_i)= \frac{n_i(n_i-1)}{2}c + \sum_{A} \sum_{
\alpha_i<\beta_i} ( h^A_{\alpha_i
\alpha_i}h^A_{\beta_i\beta_i}-(h^A_{\alpha_i\beta_i})^2)
\end{align*}
for $i=1,\ldots,k$. We see that we can assume without loss of generality that $c=0$, and
we have
\begin{equation}\begin{aligned}\label{firstineq1}
&\hskip-.2in\tau - \sum_i \tau(L_i) = \\ &\sum_A \left\{ \sum_{r<s} (h^A_{rr}h^A_{ss}-(h^A_{rs})^2) + \sum_i \sum_{\alpha_i,r} (h^A_{\alpha_i \alpha_i}h^A_{\alpha_i,r}-(h^A_{\alpha_i r)})^2) \right.  \\ 
 & \left.  + \sum_{i<j} \sum_{\alpha_i,\alpha_j} (h^A_{\alpha_i \alpha_i}h^A_{\alpha_j \alpha_j} - (h^A_{\alpha_i \alpha_j})^2) \right\}   \\
&\leq  \sum_A \left\{ \sum_{r<s} h^A_{rr}h^A_{ss} + \sum_i \sum_{\alpha_i,r} h^A_{\alpha_i \alpha_i}h^A_{rr} + \sum_{i<j} \sum_{\alpha_i,\alpha_j} h^A_{\alpha_i \alpha_i}h^A_{\alpha_j \alpha_j} \right\}  \\
 & - \sum_r \sum_{B \neq r} (h^B_{rr})^2 - \sum_i \sum_{\alpha_i} \sum_{B \notin \Delta_i} (h^B_{\alpha_i \alpha_i})^2. \end{aligned}\end{equation}
We want to prove that \eqref{firstineq1} is less than or equal to
$$ n^2 C \|H\|^2 = C \sum_A \left( \sum_B h^A_{BB}\right)^2, $$
with
\begin{equation} \label{valueC1} C = \frac{n_{k+1} + 3k - 1 - 6\sum_{i=1}^k \frac{1}{2+n_i}}
{2 \left(n_{k+1} + 3k + 2 - 6\sum_{i=1}^k \frac{1}{2+n_i} \right)}. \end{equation}
In fact, we want to prove that this value for $C$ is the best possible one in the sense that the inequality in the theorem will no longer be true in general for smaller values of $C$.

In view of Lemma \ref{lem1}, we have to find the smallest possible $C$ for which the following two statements hold:
\begin{itemize}
\item[(I)] for any $\ell \in \{1,\ldots,k\}$ and any $\gamma_{\ell} \in \Delta_{\ell}$
\begin{align*} & \sum_{r<s} h^{\gamma_{\ell}}_{rr}h^{\gamma_{\ell}}_{ss} + \sum_i \sum_{\alpha_i,r} h^{\gamma_{\ell}}_{\alpha_i \alpha_i}h^{\gamma_{\ell}}_{rr} + \sum_{i<j} \sum_{\alpha_i,\alpha_j} h^{\gamma_{\ell}}_{\alpha_i \alpha_i}h^{\gamma_{\ell}}_{\alpha_j \alpha_j} \\ & - \sum_r (h^{\gamma_{\ell}}_{rr})^2 - \sum_{i \neq \ell} \sum_{\alpha_i} (h^{\gamma_{\ell}}_{\alpha_i \alpha_i})^2 \leq C \left( \sum_B h_{BB}^{\gamma_{\ell}} \right)^2,
\end{align*}
\item[(II)] for any $t \in \Delta_{k+1}$
\begin{align*} & \sum_{r<s} h^t_{rr}h^t_{ss} + \sum_i \sum_{\alpha_i,r} h^t_{\alpha_i \alpha_i}h^t_{rr} + \sum_{i<j} \sum_{\alpha_i,\alpha_j} h^t_{\alpha_i \alpha_i}h^t_{\alpha_j \alpha_j} \\ & - \sum_{r \neq t} (h^t_{rr})^2 - \sum_i \sum_{\alpha_i} (h^t_{\alpha_i \alpha_i})^2 \leq C \left( \sum_B h_{BB}^t \right)^2.
\end{align*}
\end{itemize}

\smallskip

\noindent\emph{Step 2: Finding the best possible $C$ in \emph{(I)}.} The inequality in (I) is equivalent to
\begin{equation}\begin{aligned} \label{quadraticform1}& (C+1) \sum_{i \neq \ell} \sum_{\alpha_{\ell}} (h_{\alpha_i \alpha_i}^{\gamma_{\ell}})^2 
+ C \sum_{\alpha_{\ell}} (h_{\alpha_{\ell} \alpha_{\ell}}^{\gamma_{\ell}})^2 + (C+1) \sum_r (h_{rr}^{\gamma_{\ell}})^2\\ \\& \hskip.4in +
2C \sum_i \sum_{\alpha_i < \beta_i} h_{\alpha_i \alpha_i}^{\gamma_{\ell}}h_{\beta_i \beta_i}^{\gamma_{\ell}} \\ &
+(2C-1) \left( \sum_{i<j} \sum_{\alpha_i,\alpha_j} h_{\alpha_i \alpha_i}^{\gamma_{\ell}} h_{\alpha_j \alpha_j}^{\gamma_{\ell}} + \sum_i \sum_{\alpha_i, r} h_{\alpha_i \alpha_i}^{\gamma_{\ell}} h_{rr}^{\gamma_{\ell}} + \sum_{r<s} h_{rr}^{\gamma_{\ell}} h_{ss}^{\gamma_{\ell}}\right)\\& \geq 0.
\end{aligned}\end{equation}
If we put $x_A = h^{\gamma_{\ell}}_{AA}$ for all $A=1,\ldots,n$,
then we can look at the left hand side of the above inequality as a
quadratic form on $\R^n$. In view of Lemma \ref{lem1}, we need to find necessary and sufficient conditions on $C$ for this quadratic form to be non-negative. Two times the matrix of this quadratic
form consists of $(k+1)^2$ blocks:
$$ M_{\ell} = \left( \Lambda_{ij} \right)_{i,j=1,\ldots,k+1} $$
with
\begin{align*}
& \Lambda_{\ell\ell} = \left( \begin{array}{ccc} 2C &\cdots&
2C\\\vdots&\ddots&\vdots\\2C&\cdots& 2C \end{array}\right) \in \R^{n_{\ell}\times n_{\ell}}, \\
& \Lambda_{k+1 \, k+1} = \left( \begin{array}{ccccc} 2(C+1)&2C-1&\cdots&2C-1&2C-1 \\
2C-1&2(C+1)&\cdots&2C-1&2C-1 \\ \vdots&\vdots&\ddots&\vdots&\vdots \\
2C-1&2C-1&\cdots&2(C+1)&2C-1 \\
2C-1&2C-1&\cdots&2C-1&2(C+1)\end{array} \right) \in \R^{n_{k+1} \times
n_{k+1}},\\
& \Lambda_{ii} = \left( \begin{array}{ccccc} 2(C+1)&2C&\cdots&2C&2C \\
2C&2(C+1)&\cdots&2C&2C \\ \vdots&\vdots&\ddots&\vdots&\vdots \\
2C&2C&\cdots&2(C+1)&2C \\
2C&2C&\cdots&2C&2(C+1)\end{array} \right) \in \R^{n_i \times
n_i} \\ & \hskip1in \mbox{if} \ i \neq \ell,k+1,\\
& \Lambda_{ij} = \left( \begin{array}{ccc} 2C-1 &\cdots&
2C-1\\\vdots&\ddots&\vdots\\2C-1&\cdots&2C-1\end{array}\right)
\in \R^{n_i\times n_j} \ \mbox{if} \ i\neq j.
\end{align*}

Remark that for every $i\in \{1,\ldots,k+1\}$, $M_{\ell}$ has the
following $n_i-1$ eigenvectors:
\begin{equation} \label{eigenvecs1} \begin{aligned}
& (0,\ldots,0, | 1,-1,0,\ldots,0,0, | 0,\ldots,0), \\
& (0,\ldots,0, | 1,0,-1,\ldots,0,0, | 0,\ldots,0), \\
& \quad \vdots \\
& (0,\ldots,0, | \underbrace{1,0,0,\ldots,0,-1,}_{\Delta_i} |
0,\ldots,0).
\end{aligned} \end{equation}
The eigenvalues are $0$, $3$ or $2$ depending on whether $i=\ell$, $i=k+1$ or $i
\neq \ell,k+1$ respectively. In total we have thus found $n-(k+1)$
eigenvectors of $M_{\ell}$ with non-negative eigenvalues. The orthogonal
complement of all these eigenvectors is spanned by
\begin{equation} \label{othereigenvecs1} v_i = \frac{1}{n_i}(0,\ldots,0, |  \underbrace{1,1,\ldots,1,}_{\Delta_i} | 0,\ldots,0), \quad i=1,\ldots, k+1. \end{equation}
It is now sufficient to prove that the matrix $M'_{\ell} = (v_i
M_{\ell} v_j^T)_{i,j=1,\ldots,k+1} \in \R^{(k+1)\times (k+1)}$ is non-negative. It
follows from a direct computation that
\begin{align*}
& (M'_{\ell})_{\ell\ell} = 2C, \quad (M'_{\ell})_{k+1 \ k+1} = 2C-1+\frac{3}{n_{k+1}}, \\
& (M'_{\ell})_{ii} = 2\left(C+\frac{1}{n_i}\right) \ \mbox{if} \ i\neq\ell,k+1, \quad
(M'_{\ell})_{ij} = 2C-1 \ \mbox{if} \ i\neq j.
\end{align*}

We investigate three cases.

\smallskip
\textbf{Case 1: $2C=1$.} In this case, $M'_{\ell}$ is
a diagonal matrix with positive diagonal entries, so it is
positive definite.

\smallskip
\textbf{Case 2: $2C>1$.} In this case, the matrix
$M'_{\ell}$ is always positive definite. To see this, it is
sufficient to verify that the matrix $M''_{\ell} =
M'_{\ell}/(2C-1)$ is positive definite. By Sylvester's criterion,
we have to verify that the $(j \times j)$-matrix in the upper left
corner of $M''_{\ell}$ has positive determinant for all
$j=1,\ldots,k$. From Lemma \ref{lem2}, it is sufficient to remark that
\begin{equation} \notag  \frac{2C}{2C-1}-1 > 0, \quad
\frac{2(C+1/n_j)}{2C-1}-1 > 0 \mbox{ for all } j \in
\{1,\ldots,k\}\setminus\{\ell\}\end{equation} and \begin{equation} \notag
\frac{3}{n_{k+1}(2C-1)} > 0.\end{equation}

\smallskip
\textbf{Case 3: $2C<1$.} It is now sufficient to require that $M''_{\ell} = M'_{\ell}/(2C-1)$ is non-positive. By
Sylvester's criterion, this is equivalent to the determinant of the
$(j \times j)$-matrix in the upper left corner of $M''_{\ell}$
having sign $(-1)^j$ for all $j=1,\ldots,k+1$. It follows from Lemma \ref{lem2} that these determinants are
\begin{align*}
& D_j = \frac{1}{(2C-1)^{j-1}} \left( \frac{(2C)^{\delta_j}}{2C-1} + \sum_{i=1 \atop i \neq \ell}^j \frac{n_i}{n_i+2} \right) \prod_{i=1 \atop i \neq \ell}^j \left( 1 + \frac{2}{n_i} \right) \ \mbox{for $j = 1,\ldots,k$,} \\
& D_{k+1} =  \frac{3}{n_{k+1}(2C-1)^k} \left( \frac{2C}{2C-1} + \frac{n_{k+1}}{3} + \sum_{i=1 \atop i \neq \ell}^k \frac{n_i}{n_i+2} \right) \prod_{i=1 \atop i \neq \ell}^k \left( 1 + \frac{2}{n_i} \right),
\end{align*}
where $\delta_j=0$ if $j<\ell$ and $\delta_j=1$ if $j\geq\ell$. Hence, we have
\begin{align*}
& \mathrm{sgn}(D_j) = (-1)^{j-1} \mathrm{sgn} \left( \frac{(2C)^{\delta_j}}{2C-1} + \sum_{i=1 \atop i \neq \ell}^j \frac{n_i}{n_i+2} \right) \ \mbox{for $j = 1,\ldots,k$,} \\
& \mathrm{sgn}(D_{k+1}) = (-1)^k \mathrm{sgn} \left( \frac{2C}{2C-1} + \frac{n_{k+1}}{3} + \sum_{i=1 \atop i \neq \ell}^k \frac{n_i}{n_i+2} \right),
\end{align*}
and the conditions for $M''_{\ell}$ to be non-positive are
\begin{equation} \label{firstcondC1}
\left\{\begin{array}{l}
\displaystyle{ \frac{(2C)^{\delta_j}}{2C-1} + \sum_{i=1 \atop i \neq \ell}^j \frac{n_i}{n_i+2} \leq 0 \quad \mbox{for} \ j = 1,\ldots,k,} \\
\displaystyle{ \frac{2C}{2C-1} + \frac{n_{k+1}}{3} + \sum_{i=1 \atop i \neq \ell}^k \frac{n_i}{n_i+2} \leq 0. }
\end{array}\right.
\end{equation}
One can verify that the last inequality implies the first $k$ inequalities, and the last inequality is equivalent to
\begin{equation} \label{firstineqC1}
2C \geq \frac{n_{k+1} + 3k - 3 - 6 \sum_{i\neq\ell}\frac{1}{n_i+2}}{n_{k+1} + 3k - 6 \sum_{i\neq\ell}\frac{1}{n_i+2}}.
\end{equation}
We conclude from all three cases that the quadratic form in \eqref{quadraticform1} is non-negative if and only if $C$ satisfies \eqref{firstineqC1} for every $\ell = 1,\ldots,k$.
\smallskip

\noindent\emph{Step 3: Finding the best possible $C$ in \emph{(II)}.} We can proceed in the same way as in Step 2, by defining a quadratic form on $\R^n$ from inequality (II) and looking for the best possible value of $C$ for which this quadratic form is non-negative. Since the result is the same as the one already obtained in \cite{CD} by a more ad hoc method, we will not go into details here. The condition on $C$ is 
\begin{equation} \label{secondineqC1}
2C \geq \frac{n_{k+1} + 3k - 1 - 6 \sum_{i}\frac{1}{n_i+2}}{n_{k+1} + 3k + 2 - 6 \sum_{i}\frac{1}{n_i+2}}. 
\end{equation}
Since the right hand side of \eqref{firstineqC1} is less than the right hand side of \eqref{secondineqC1} (for any $\ell \in \{1,\ldots,k\}$), we only have \eqref{secondineqC1} as a condition on $C$ and we conclude that the best possible value for $C$ is precisely the value given in \eqref{valueC1}.

\smallskip

\noindent\emph{Step 4: The equality case.} Assume that equality holds at a point. Then one has to have equality in \eqref{firstineq1}, which implies precisely the first condition given in the theorem. Next, one also has to have equality in (I), which implies that the vector $(h^{\alpha_i}_{11},\ldots,h^{\alpha_i}_{nn})$ has to be a linear combination of the vectors given in \eqref{eigenvecs1} for every $i=1,\ldots,k$ and every $\alpha_i\in\Delta_i$. (Remark that due to the choice of $C$, the quadratic form \eqref{quadraticform1} is positive definite on the orthogonal complement of these vectors.) Hence, one obtains exactly the second condition given in the theorem. Finally, one has to have equality in (II). We refer to \cite{CD} to see that this is equivalent to the third condition.
\end{proof}

\begin{remark}
In inequality \eqref{firstineq1}, we omit exactly the squares of components of $h$ with three different indices. Besides the technique used to prove non-negativeness of
the quadratic forms, this is the main difference with the proof in \cite{CD}, where also terms of type $-(h_{\alpha_i \alpha_i}^{\alpha_j})^2$ are omitted, making
the inequality less sharp.
\end{remark}

\begin{theorem} \label{theo2}
Let $M^n$ be a Lagrangian submanifold of a complex space form
$\tilde M^n(4c)$. Let $n_1,\ldots,n_k$ be integers satisfying $2\leq
n_1 \leq \ldots \leq n_k \leq n-1$ and $n_1+\ldots+n_k=n$. Then, at any point of $M^n$, the following holds:
\begin{equation}\begin{aligned}\notag& \delta(n_1,\ldots,n_k)
\leq \frac{n^2\left(k-1-2\sum_{i=2}^k\frac{1}{n_i+2}\right)}{2\left(k-2\sum_{i=2}^k\frac{1}{n_i+2}\right)}\|H\|^2
\\& \hskip.4in +\frac{1}{2}\left(n(n-1)-\sum_{i=1}^k n_i(n_i-1)\right)c.
\end{aligned}\end{equation}

Assume that equality holds at a point $p \in M^n$. Then with the choice of basis and the notations introduced at the beginning of this section, one has
\begin{itemize}
\item $h_{\alpha_i \alpha_j}^A = 0$ if $i \neq j$ and $A \neq \alpha_i,\alpha_j$,
\item if $n_j \neq \min\{n_1,\ldots,n_k\}$: 
$$h_{\alpha_i \alpha_i}^{\beta_j}=0 \mbox{ if } i \neq j \mbox{ and } \sum_{\alpha_j \in \Delta_j} h_{\alpha_j \alpha_j}^{\beta_j} = 0,$$
\item if $n_j = \min\{n_1,\ldots,n_k\}$: 
$$\sum_{\alpha_j \in \Delta_j} h_{\alpha_j \alpha_j}^{\beta_j} = (n_i+2) h_{\alpha_i \alpha_i}^{\beta_j} \mbox{ for any } i \neq j \mbox{ and any } \alpha_i \in \Delta_i.$$
\end{itemize}
\end{theorem}

\begin{remark}
In the case of equality, we don't have information about $h_{\alpha_i \beta_i}^{\gamma_i}$, where $\alpha_i$, $\beta_i$ and $\gamma_i$ are mutually different indices in the same block $\Delta_i$.
\end{remark}

\begin{proof} The set-up of the proof is exactly the same as in the previous case, but now $\Delta_{k+1}=\varnothing$ and hence $n_{k+1}=0$.
We now have
\begin{eqnarray}
\tau - \sum_i \tau(L_i) &=& \sum_A \sum_{i<j}
\sum_{\alpha_i,\alpha_j} (h^A_{\alpha_i \alpha_i}h^A_{\alpha_j
\alpha_j} - (h^A_{\alpha_i \alpha_j})^2) \nonumber \\
&\leq& \sum_A \sum_{i<j} \sum_{\alpha_i,\alpha_j} h^A_{\alpha_i
\alpha_i}h^A_{\alpha_j \alpha_j} - \sum_i \sum_{\alpha_i} \sum_{B
\notin \Delta_i} (h^B_{\alpha_i \alpha_i})^2 \label{firstineq2}
\end{eqnarray}
and we want to prove that the best possible value of $C$ for which \eqref{firstineq2} is less than or equal to
$$ n^2 C \|H\|^2 = C \sum_A \left( \sum_B h^A_{BB}\right)^2, $$
is
\begin{equation} \label{valueC2} C = \frac{k-1-2\sum_{i=2}^k\frac{1}{n_i+2}}
{2\left(k-2\sum_{i=2}^k\frac{1}{n_i+2}\right)}. \end{equation}

We only have to consider inequality (I), which reduces to: for all $\ell \in \{1,\ldots,k\}$ and all $\gamma_{\ell} \in \Delta_{\ell}$
\begin{equation} \label{ineqI2} 
\sum_{i<j} \sum_{\alpha_i,\alpha_j} h_{\alpha_i
\alpha_i}^{\gamma_{\ell}} h_{\alpha_j \alpha_j}^{\gamma_{\ell}} -
\sum_{i\neq\ell} \sum_{\alpha_i}(h_{\alpha_i
\alpha_i}^{\gamma_{\ell}})^2 \leq C \left( \sum_B
h^{\gamma_{\ell}}_{BB}\right)^2, 
\end{equation}
or, equivalently,
\begin{equation}\begin{aligned} &\notag (2C-1) \sum_{i<j} \sum_{\alpha_i,\alpha_j} x_{\alpha_i}
x_{\alpha_j} +
2C \sum_i \sum_{\alpha_i<\beta_i} x_{\alpha_i} x_{\beta_i} \\& \hskip.3in +
(C+1) \sum_{i \neq \ell} \sum_{\alpha_i} x_{\alpha_i}^2 + C \sum_{\alpha_{\ell}}
x_{\alpha_{\ell}}^2 \geq 0\end{aligned}\end{equation}
for all $(x_1,\ldots,x_n)\in\R^n$, where we have put $h^{\gamma_{\ell}}_{AA}=x_A$ for $A=1,\ldots,n$ as before.
Requiring non-negativeness of this quadratic form can be done in exactly the same way as before, considering now only the 
$k^2$ blocks in the upper left corner of $M_{\ell}$ and therefore only 
giving the first $k$ conditions of \eqref{firstcondC1}. It is clear that the $k$th condition implies all the other ones and hence the necessary and sufficient
condition on $C$ for the quadratic form to be non-negative is
$$ \frac{2C}{2C-1} + \sum_{i=1 \atop i \neq \ell}^k \frac{n_i}{n_i+2} \leq 0 \mbox{ for } \ell = 1,\ldots,n. $$
Since $n_1 \leq n_2 \leq \ldots \leq n_k$, the above condition for $\ell = 1$ implies all the others and we obtain
$$ 2C \geq \frac{k-1-2\sum_{i=2}^k \frac{1}{n_i+2}}{k-2\sum_{i=2}^k \frac{1}{n_i+2}}, $$
such that the best possible value for $C$ is indeed the one given in \eqref{valueC2}.

Now let us investigate when equality is attained in the inequality we just proved. In order for this to happen, one has to have equality in \eqref{firstineq2} and one has to have equality in \eqref{ineqI2} for all $\ell \in \{1,\ldots,k\}$ and all $\gamma_{\ell} \in \Delta_{\ell}$.

From equality in \eqref{firstineq2}, we conclude that $h_{\alpha_i \alpha_j}^A=0$ for $i \neq j$ and $A \neq \alpha_i, \alpha_j$.

If equality holds in \eqref{ineqI2}, the vector $(h_{11}^{\gamma_{\ell}}, \ldots, h_{nn}^{\gamma_{\ell}})$ has to be in the kernel of $M_{\ell}$. If $n_{\ell} \neq \mathrm{min}\{n_1,\ldots,n_k\}$ it follows from the proof of Theorem \ref{theo1} that this kernel is spanned by the vectors \eqref{eigenvecs1} (with $i = \ell$). This corresponds to the first possibility given in the theorem.  If $n_{\ell} = \mathrm{min}\{n_1,\ldots,n_k\}$, it follows from the proof of Theorem \ref{theo1} that the kernel of $M_{\ell}$ is larger due to the choice of $C$. In particular, there will be non-zero linear combinations of the vectors $v_1,\ldots,v_k$, given in \eqref{othereigenvecs1}, in the kernel. Assume that
$$ M_{\ell} \left( \sum_{i=1}^k a_i v_i \right) = 0 $$
for some real numbers $a_1,\ldots,a_k$. A straightforward computation shows that this is equivalent to
$$ \left\{ \begin{array}{l} \displaystyle{\forall i \in \{1,\ldots,k\}\setminus\{\ell\}: \left(\sum_{j=1}^k a_j\right)(2C-1) + a_i\left(1+\frac{2}{n_i}\right) = 0,} \\
\displaystyle{\left(\sum_{j=1}^k a_j\right)(2C-1) + a_{\ell} = 0.} \end{array} \right. $$
One can check that this system indeed has a non-zero solution for $(a_1,\ldots,a_k)$ if and only if
$$ C = \frac{k-1-2\sum_{i \neq \ell} \frac{1}{n_i+2}}{2 \left( k-2\sum_{i \neq \ell} \frac{1}{n_i+2} \right)}, $$
i.e., if and only if $n_{\ell}=\min\{n_1,\ldots,n_k\}$. In this case, the solution is given by
$$ a_i=\frac{\lambda n_i}{n_i+2} \mbox{ for } i\neq\ell \mbox{ and } a_{\ell}=\lambda $$
for some real number $\lambda$. The vector
$(h_{11}^{\gamma_{\ell}}, \ldots, h_{nn}^{\gamma_{\ell}})$ thus
has to satisfy the conditions given in the last possibility in the
statement of the theorem.
\end{proof}

\end{document}